\documentclass{amsart}
\usepackage{amssymb}
\usepackage{verbatim}
\usepackage{array}
\usepackage{latexsym}
\usepackage{enumerate}
\usepackage{amsmath,mathtools}
\usepackage{amsfonts}
\usepackage{amsthm}
\usepackage[foot]{amsaddr}
\usepackage{color}
\usepackage[english]{babel}
\usepackage{graphicx}

\usepackage{hyperref}
\usepackage[capitalise]{cleveref}

\newtheorem{theorem}{Theorem}[section]
\newtheorem{proposition}[theorem]{Proposition}
\newtheorem{lemma}[theorem]{Lemma}
\newtheorem{corollary}[theorem]{Corollary}
\newtheorem{definition}[theorem]{Definition}

\newtheorem{remark}[theorem]{Remark}

\def\F{\mathbf F}

\def\Z{\mathbf Z}

\def\min{{\rm min}}

\newcommand{\aut}{\mbox{\rm Aut}}

\newcommand{\PP}{\mathcal{P}}
\newcommand{\QQ}{\mathcal{Q}}
\newcommand{\RR}{\mathcal{R}}

\newcommand{\cc}{\mathfrak{c}}
\newcommand{\dd}{\mathfrak{d}}

\renewcommand{\PP}{\mathcal{P}}
\renewcommand{\QQ}{\mathcal{Q}}
\renewcommand{\F}{\mathbb{F}}
\renewcommand{\Z}{\mathbb{Z}}

\renewcommand{\aut}{\mathrm{Aut}}

\newcommand{\Zte}{\overline{Z}_3}

\title[On a maximal function field with the third largest possible genus, $q \equiv 0 \pmod 3$]{Weierstrass semigroups and automorphism group of a maximal function field with the third largest possible genus, $q \equiv 0 \pmod 3$}
\date{}

\author[Peter Beelen]{Peter Beelen$^1$} \address{$^1$Department of Applied Mathematics and Computer Science, Technical University of Denmark, Kongens Lyngby 2800, Denmark} \email{pabe@dtu.dk} \thanks{}
\author[Maria Montanucci]{Maria Montanucci$^1$} \address{} \email{marimo@dtu.dk} \thanks{}
\author[Lara Vicino]{Lara Vicino$^2$} \address{$^2$Faculty of Science and Engineering - Bernoulli Institute, University of Groningen, Groningen 9747AG, The Netherlands}  \email{l.vicino@rug.nl} \thanks{}

\begin{document}

\begin{abstract}
In this article we complete the work started in \cite{BMV} and \cite{BMV2}, explicitly determining the Weierstrass semigroup at any place and the full automorphism group of a known $\mathbb{F}_{q^2}$-maximal function field $Z_3$ having the third largest genus, for $q \equiv 0 \pmod 3$. The cases $q \equiv 2 \pmod 3$ and $q \equiv 1 \pmod 3$ have been in fact analyzed in \cite{BMV} and \cite{BMV2}, respectively. As in the other two cases, the function field $Z_3$ arises as a Galois subfield of the Hermitian function field, and its uniqueness (with respect to the value of its genus) is a well-known open problem. Knowing the Weierstrass semigroups may provide a key towards solving this problem. Surprisingly enough, $Z_3$ has many different types of Weierstrass semigroups and the set of its Weierstrass places is much richer
than its set of $\mathbb{F}_{q^2}$-rational places. We show that a similar exceptional behaviour does not occur in terms of
automorphisms, that is, $\aut(Z_3)$ is exactly the automorphism group inherited from the Hermitian function field, apart from the case $q=3$.
\end{abstract}

\maketitle



\section{Introduction}

An algebraic function field $F$ over a finite field with square cardinality is called maximal if the Hasse-Weil upper bound is attained with equality. If $F$ is an algebraic function field of genus $g$ over the finite field $\mathbb{F}_{q^2}$ of cardinality $q^2$ ($q$ prime power),
then the Hasse-Weil upper bound states that
$$N(F) = q^2 + 1 + 2qg,$$
where $N(F)$ denotes the number of places of degree one (we will refer to them also as $\mathbb{F}_{q^2}$-rational places) of
$F$.

$\mathbb{F}_{q^2}$-maximal function fields, in particular those with large genus, have been and still are intensively investigated in the literature.
This is also because of their applications in coding theory and cryptography, due to Goppa’s construction of error-correcting
codes using algebraic function fields (AG codes). Two questions arise naturally in the literature about algebraic function field: \textit{how large can the genus of a maximal function field be? Is it possible to classify (up to isomorphism) all maximal function fields over a finite field of given genus $g$?} 

Some properties about the genera spectrum of maximal function fields are known. It is for example well known that if $g$ denotes the genus of an $\mathbb{F}_{q^2}$-maximal function field  $F$, then $g  \leq  q(q - 1)/2$, see \cite{Y}, and that $g$ reaches this upper limit if and only if $F$ is isomorphic to the Hermitian function field $H_1 := \mathbb{F}_{q^2}(u, v)$ with $v^{q+1} = u^q + u$, see \cite{RS}. 

In \cite{FGT} it has been further proven that either $g \leq \lfloor (q-1)^2/4\rfloor$, or $g = q(q-1)/2$. For $q$ odd, $g = (q-1)^2/4$ occurs if and only if $F$ is isomorphic to the function field $H_2 := \mathbb{F}_{q^2}(x,y)$ with $y^q +y = x^{(q+1)/2}$, see \cite[Theorem 3.1]{FGT}. For $q$ even, a similar result is obtained in \cite{AT}: $g = \lfloor(q - 1)^2/4\rfloor = q(q- 2)/4$ occurs if and only if $F$ is isomorphic to the the function field $H_2 :=\mathbb{F}_{q^2}(x, y)$ with $y^{q/2} + \ldots + y^2 + y = x^{q+1}$.

This means that both the first $g_1 := q(q - 1)/2$ and the second $g_2 := \lfloor (q - 1)^2/4\rfloor$ largest genera of $\mathbb{F}_{q^2}$-maximal
function fields are known, and they are realized by exactly one function field up to isomorphism, but the same
is not known for the third largest genus. In fact, what is known is that its value is $g_3 := \lfloor (q^2 -q +4)/6\rfloor$, but it is still
unclear whether this is realized by exactly one function field up to isomorphism or not, see \cite{KT}. It has been shown in \cite[Remark 3.4]{KT} that an $\mathbb{F}_{q^2}$-maximal function field of genus $g_3$ does
exist, namely
\begin{enumerate}
    \item  $X_3 := \mathbb{F}_{q^2}(x, y)$ with $x^{(q+1)/3} + x^{2(q+1)/3} + y^{q+1} = 0$, if $q \equiv 2 \pmod 3$;
    \item $Y_3 := \mathbb{F}_{q^2}(x, y)$ with $x^q+xy^{(q-1)/3}-y^{(2q+1)/3} = 0$, if $q \equiv 1 \pmod 3$; and
    \item $Z_3 := \mathbb{F}_{q^2}(x, y)$ with $x^q+x+(y+y^3+\cdots+y^{q/3})^2=0$, if $q \equiv 0 \pmod 3$.
\end{enumerate}

All the examples above are Galois subfields of the Hermitian function field $H_1$ and the corresponding Galois group has order three. 
Understanding if a uniqueness result (up to isomorphism) is true also in the case of genus $g_3$, is a famous open problem. 

In the proof of uniqueness for the cases $g_1$ and $g_2$ mentioned above, the so-called Weiestrass semigroups and Weierstrass places play a central role, so it is natural to ask how these behave for the three function fields mentioned above. The Weiestrass points and Weierstrass semigroups at every places of $X_3$ and $Y_3$ have been computed in \cite{BMV} and \cite{BMV2}, respectively. Continuing and completing our investigation, in this paper we compute the Weierstrass semigroup at every place of the function field $Z_3$ having third largest genus $g(Z_3)=g_3$ for $q \equiv 0 \pmod 3$, as well as its set of Weierstrass places and its full automorphism group. 

Doing so, we show that $Z_3$ has a quite large set of non-rational Weierstrass places and many different types of Weierstrass semigroups. The full automorphism group of $Z_3$ is also computed as an application of the results mentioned above.
The paper is organized as follows: we start by giving basic properties of the function field $Z_3$ and use them to compute Weierstrass semigroups for some of the rational places of $Z_3$. Then, in the remainder of the second section, we construct two families of functions that will be used later on. In the third section we compute the Weierstrass semigroups for all places. Finally, in the last section, we characterize all automorphisms of $Z_3$. 

\section{Preliminary results}

In this section, we provide some preliminary results that will be used throughout the paper. In the first subsection, the function field $Z_3$ is presented and some of its initial properties are pointed out. Among other things, some principal divisors are computed, as well as some differentials and the corresponding canonical divisors on $Z_3$. In fact, canonical divisors and differentials are useful for computing gaps, and hence Weierstrass semigroups. Local power series expansions for some useful functions can also be found there. In the second subsection
we instead start our investigation on Weierstrass semigroups of $Z_3$ by computing $H(P)$ at some special places $P$.

\subsection{The function field $Z_3$ and its properties}

Let $q=3^t$ with $t \geq 1$. The function field $Z_3:=\mathbb{F}_{q^2}(x,y)$ is given by the equation 
$$x^q+x+(y+y^3+\cdots+y^{q/3})^2=0.$$
It is a maximal function field of genus $q(q-1)/6$. The function field $Z_3$ can be viewed as a subfield of the Hermitian function field $H_1=\mathbb{F}_{q^2}(u,v)$ where $u^q+u=v^{q+1}$, using the equations $x=-(u+v^2)$ and $y=v^3-v$. With this interpretation, $Z_3$ is the fixed field of the automorphism of $H_1$ defined by $(u,v) \mapsto (u+v+2,v+1).$ In particular, the extension $\mathbb{F}_{q^2}(u,v)/\mathbb{F}_{q^2}(x,y)$ is a Galois extension of degree three. In this extension, there is one totally ramified place of the Hermitian function field, namely the pole of $u$ and $v$ which we denote by $Q_\infty$. We denote the place of $Z_3$ lying below $Q_\infty$ by $P_\infty$. The function field $\Zte:=\overline{\mathbb{F}}_{q^2}Z_3$ is obtained from $Z_3$ by extending the constant field to $\overline{\mathbb{F}}_{q^2}$, the algebraic closure of ${\mathbb{F}}_{q^2}$. Since $P_\infty$ is a rational place of $Z_3$, it gives rise to a unique place of $\Zte$, which we again denote by $P_\infty$. A place $P$ of $\Zte$ distinct from $P_\infty$ does not have a pole in $x$ nor in $y$. In particular, such a place $P$ has a unique $x$- and $y$-coordinate, that is to say $x(P)=a$ and $y(P)=b$  for certain $a,b \in \overline{\mathbb{F}}_{q^2}$ satisfying $a^q+a+(b+b^3+\cdots+b^{q/3})^2=0$. Conversely, any pair $(a,b) \in \overline{\mathbb{F}}_{q^2}^2$ satisfying $a^q+a+(b+b^3+\cdots+b^{q/3})^2=0$ corresponds to a unique place $P$ of $\Zte$ with $x$- and $y$-coordinate $a$ and $b$. We will sometimes denote that place by $P_{(a,b)}$. In case $(a,b) \in \mathbb{F}_{q^2}^2$, we will with slight abuse of notation denote by $P_{(a,b)}$ the rational place of $Z_3$ with $x$-coordinate $a$ and $y$-coordinate $b$. 

It is not hard to compute the divisors of the functions $x$ and $y$. They are as follows:
\begin{equation} \label{divx1}
(x)=2\sum_{p(b)=0} P_{(0,b)}-2m P_\infty,
\end{equation}
where $p(b)=b+b^3+\cdots+b^{q/3}$ and $m=q/3$.
Moreover
\begin{equation} \label{divy1}
(y)=\sum_{a^q+a=0}P_{(a,0)}-q P_\infty.
\end{equation}

For future computations, we will also need a suitable canonical divisor and information about possible gaps of a place. We state these in the next lemma.
\begin{lemma}\label{lem:diffandgaps}
The divisor of the differential $dy$ of $Z_3$ is given by
\begin{equation*} 
(dy)=(2g(Z_3)-2)P_\infty=\frac{q^2-q-6}{3}P_\infty=(m-1)(q+2)P_\infty.
\end{equation*}    
For any function $f \in L((m-1)(q+2)P_\infty)$ and any place $P$ of $\Zte$ different from $P_\infty$, the value $v_P(f) + 1$ is a gap of $P$, that is, $v_P(f) + 1 \in G(P)$.
\end{lemma}

\begin{proof}
Denote as before by $H_1=\mathbb{F}_{q^2}(u,v)$ the Hermitian function field where $u^q+u=v^{q+1}$. As recalled previously, the extension $\mathbb{F}_{q^2}(u,v)/\mathbb{F}_{q^2}(x,y)$ is a Galois extension with only one ramified place, which is totally ramified, namely the common pole of $u$ and $v$ which we denote by $Q_\infty$ (we denote as before with $P_\infty$ the place below $Q_\infty$ in this Galois extension). From \cite[Corollary 3.4.7]{Sti} the differential $dv$ on $H_1$ has divisor $(dv)=(q(q-1)-2)Q_\infty$. From \cite[Theorem 3.4.6]{Sti} only $P_\infty$ can be in the support of $(dy)_{Z_3}$ and since this divisor is canonical, it has degree $2g(Z_3)-2$. Hence $(dy)_{Z_3}=(2g(Z_3)-2)P_\infty=(m-1)(q+2)P_\infty$.

The same argument shows $(dy)_{\Zte}=(m-1)(q+2)P_\infty$. This shows that, for any $f \in L((m-1)(q+2)P_\infty)$, the differential $fdy$ is holomorphic, that is, it has no poles. Then by \cite[Corollary 14.2.5]{VS} the integer $v_P(fdy) + 1 \in G(P)$, for any place $P$ of $\Zte$ different from $P_\infty$, as claimed.
\end{proof}

Finally, we recall the Fundamental Equation mentioned in for example \cite[Page xix (ii)]{HKT}, which implies that there exists for any place $P$ of $\Zte$ a function $F_P$ in $\Zte$ such that 
\begin{equation} \label{fundeq1}
(F_P)=qP+\Phi(P)-(q+1)P_\infty,
\end{equation}
where $\Phi$ denotes the $q^2$-th power map, also known as the Frobenius map. Note that in case $P$ is $\mathbb{F}_{q^2}$-rational, one can find a function $F_P \in Z_3$ such that $(F_P)=(q+1)(P-P_\infty).$

A very important ingredient in computing Weierstrass semigroups are local power series at specific places.
We define 
\begin{equation*} 
\beta(P):=\begin{cases}
    (b+b^3+\cdots+b^{q/3})^2=-(a^q+a) & \quad \mbox{if} \quad P=P_{(a,b)}, \ \mbox{on} \ \Zte\\
    \infty & \quad \mbox{if} \quad P = P_\infty.
\end{cases}
\end{equation*}
Further, if $P_{(a,b)}$ lies below a place $Q_{(A,B)}$ of the Hermitian function field, then we define $T:=(v-B)/(B^q-B).$ Here we assume that $B^q-B \neq 0$. In fact note that $B^q-B=b+b^3+\cdots+b^{q/3}=p(b)$, since $b=B^3-B$. The case $p(b)=0$ will be analyzed in Theorem \ref{semp00}, without the use of power series computations. Note that $T$ is a local parameter for the place $Q_{(A,B)}$. Directly from the defining equation of the Hermitian function field, one obtains that $u-A=B^q(v-B)+O(T^q)$, see also \cite[Page 4]{BMV2}. Here, for $n \in \Z$, the expression $O(T^n)$ denotes a function of the Hermitian function field with valuation at least $n$ at $Q_{(A,B)}$. We define $x_a:=-(x-a)/\beta(P)$ and $y_b:=-(y-b)/(B^q-B)$. Then 
\begin{eqnarray}\label{eq:xa_powerseries}
x_a & = & \frac{u+v^2-A-B^2}{\beta(P)} = \frac{(u-A)+(v-B)^2+2B(v-B)}{\beta(P)} \notag\\
 & = & \frac{(B^q-B)(v-B)+(v-B)^2}{(B^q-B)^2}+O(T^q)=T+T^2+O(T^q)
\end{eqnarray}
and
\begin{eqnarray}\label{eq:yb_powerseries}
y_b & = & \frac{v-v^3-B+B^3}{B^q-B} = \frac{(v-B)-(v-B)^3}{B^q-B} =  T-\beta(P) T^3+O(T^q).
\end{eqnarray}

For future use, let us determine the possible values of $\beta(P)$ for the rational places $P$ of $Z_3$. 

\begin{lemma}\label{lem:beta_rational}
Let $P$ be a rational place of $Z_3$. Then either $\beta(P) \in \{0,1,\infty\}$ or $\beta(P)^{(q-1)/2}=-1$. Conversely, if $P$ is a place of $Z_3$ such that $\beta(P) \in \{0,1,\infty\}$ or $\beta(P)^{(q-1)/2}=-1$, then $P$ is a rational place.
\end{lemma}
\begin{proof}
Consider the group $$G:=\{(x,y)\mapsto (x+a,\pm y+b) \mid a^q+a=0, \quad p(b)=0\}$$
of automorphisms of $Z_3$. It is a group of order $2q^3/3$. Its fixed field clearly contains $\beta=-(x^q+x)=p(y)^2$. Conversely, 
$$[Z_3:\mathbb{F}_{q^2}(\beta)]=[Z_3:\mathbb{F}_{q^2}(x)][\mathbb{F}_{q^2}(x):\mathbb{F}_{q^2}(\beta)]=2q/3 \cdot q=2q^2/3.$$ Hence the fixed field of $G$ equals $\mathbb{F}_{q^2}(\beta)$ and the extension $Z_3/\mathbb{F}_{q^2}(\beta)$ is a Galois extension of degree $2q^2/3$. 

All automorphisms in $G$ fix $P_\infty$ and therefore $P_\infty$ is the only place of $Z_3$ lying above the pole of $\beta$. There are exactly $q^2/3$ places lying above the zero of $\beta$, namely the rational places $P_{(a,b)}$, where $a^q+a=p(b)=0$. Note that indeed all solutions to the equation $p(b)=0$ are in $\mathbb{F}_{q^2}$, actually even in $\mathbb{F}_q$, since $p(b)$ gives the trace from $\mathbb{F}_q$ to $\mathbb{F}_3$. Finally, note that there are exactly $2q^2/3$ places lying above the zero of $\beta-1$, namely the rational places $P_{(a,b)}$, where $a^q+a=1$ and $p(b)=\pm 1$.

So far we have shown the lemma in case $\beta(P) \in \{0,1,\infty\}$. As we have seen, the number of possibilities for such $P$ equals $1+q^2/3+2q^2/3=1+q^2$.
Since $Z_3$ is maximal, it has precisely $1+q^2+(q^3-q^2)/6$ many rational places. We claim that the lemma follows if we can show that, for any of the $(q^3-q^2)/6=(q-1)/2 \cdot 2q^2/3$ rational places not analyzed thus far, it holds that $\beta(P)^{(q-1)/2}=-1$. Indeed, this would directly show the first part of the lemma. Moreover, since at most $2q^2/3$ places $P$ can have the same value $\beta(P)$, this would also show that for each $\gamma$ satisfying $\gamma^{(q-1)/2}=-1$, there exists at least one rational place of $Z_3$ such that $\beta(P)=\gamma$. Since $Z_3/\mathbb{F}_{q^2}(\beta)$ is a Galois extension this implies that any place $P$ such that $\beta(P)=\gamma$ is rational. 

Now let $P=P_{(a,b)}$ be a rational place of $Z_3$ such that $\beta(P) \not\in \{0,1,\infty\}$. In particular, this implies that $p(b) \not \in \mathbb{F}_3$. Since $\beta(P)=-(a^q+a)$, we see $\beta(P) \in \mathbb{F}_q$ and therefore $p(b^q)^2=-(a+a^q)^q=-(a+a^q)=p(b)^2$, which implies that either $p(b^q)=p(b)$ or $p(b^q)=-p(b)$. 

In the first case, $p(b^q)=p(b)$, we obtain that $p(b^q-b)=0$ and hence $b^q-b \in \mathbb{F}_q$. On the other hand, $(b^q-b)^q=b^{q^2}-b^q=b-b^q=-(b^q-b)$. Hence this case can only occur if $b^q-b=0$, implying that $b \in \mathbb{F}_q$. This immediately implies that $p(b) \in \mathbb{F}_3$. The first case has therefore already been analyzed before and needs not be considered here.

In the second case, $p(b^q)=-p(b)$, we obtain that $b+b^3+\cdots+b^{q^2/3}=0$. Since the left-hand side of this equation is the trace from $\mathbb{F}_{q^2}$ to $\mathbb{F}_3$, it has $q^2/3$ many solutions in $\mathbb{F}_{q^2}$, including the $q/3$ solutions to the equation $p(b)=0$. Disregarding these $q/3$ solutions, we see that $\beta(P)^{(q-1)/2}=p(b)^{q-1}=-1.$ This is what was left to show. 
\end{proof}
In the next section we compute the Weierstrass semigroups at rational places $P$ of $Z_3$ such that $\beta(P) \in \{0,1,\infty\}$.

\subsection{Computation of $H(P)$ for some special places $P$ of $Z_3$}
Our first aim is to compute the Weierstrass semigroups at the places $P$ of $Z_3$ satisfying $\beta(P) \in \{0,1,\infty\}$. We start with the common pole of $x$ and $y$, which we denoted in the previous section by $P_\infty$.

\begin{theorem} \label{sempinf}
The Weierstrass semigroup at $P_\infty$ is as follows: $$H(P_\infty)=\langle 2q/3,q,q+1 \rangle.$$
\end{theorem}

\begin{proof}
Since $Z_3$ is maximal, we know that $q$ and $q+1$ are in $H(P_\infty)$. From Equation \eqref{divx1}, we conclude that $2q/3\in H(P_\infty)$. Let us denote by $H=\langle 2q/3,q,q+1 \rangle$ the numerical semigroup generated by $2q/3,q,q+1$. Since $H \subseteq H(P_\infty)$, the number of gaps of $H$, which we denote by $g$, is at least $g(Z_3)$. Therefore it is enough to show that $g \le g(Z_3)$. However, this is a direct consequence of \cite[Korollar A.2.3]{Fuhrmann_1995}, where semigroups of the form $\langle a,q,q+1 \rangle$ were studied. Also, see \cite[Lemma 3.4]{CKT}, where some results from \cite{Fuhrmann_1995} were paraphrased. 
\end{proof}

We now wish to proceed with analyzing the Weierstrass semigroup at places satisfying $\beta(P)=0$. These places are precisely the zeros of the function $p(y)$ (equivalently $x^q+x$).

\begin{theorem} \label{semp00}
Let $a,b \in \mathbb{F}_{q^2}$ satisfy $a^q+a=0$ and $p(b)=0$. Then 
$H(P_{(a,b)})=\langle q,q+1,q-1,2q-4\rangle$. 
\end{theorem}

\begin{proof}
Let $G$ be the group of automorphisms of $Z_3$ that was introduced in the proof of Lemma \ref{lem:beta_rational}. The group $G$ acts on the set of rational places of $Z_3$, and under this action $P_{(0,0)}$ and $P_{(a,b)}$ lie in the same orbit. Hence it is enough to prove the theorem in case $a=b=0$.

By Equation \eqref{fundeq1}, there exists a function $F$ such that $(F)=(q+1)(P_{(0,0)}-P_\infty)$. Now considering the functions $x/F,y/F,1/F$ and $x^3/F^2$ and using Equations \eqref{divx1} and \eqref{divy1}, we see that $q-1$, $q$, $q+1$ and $2q-4$ are in $H(P_{(0,0)})$. 

To prove the theorem, it is enough to show that the semigroup $\langle q-1,q,q+1,2q-4\rangle$ has at most $g(Z_3)$ many gaps. Later, in Theorem \ref{thm:rat}, we will show this for a class of semigroups including the present one, which can be obtained by putting $i=1$ in Theorem \ref{thm:rat}. Therefore we will postpone the proof till then to avoid unnecessary repetition of arguments. 
\end{proof}

In the following theorem, we now determine the Weierstrass semigroup $H(P)$ in the case $\beta(P)=1$. 

\begin{theorem} \label{sembeta1}
Let $P$ be a rational place of $Z_3$ satisfying $\beta(P)=1$. Then 
\begin{equation*}
    H(P) = \langle q, q+1, (q-1) +j(q-2) \mid j=0,\ldots, m-1\rangle.
\end{equation*}
\end{theorem}
\begin{proof}
Let $P$ satisfy $\beta(P)=1$. Since $P$ is a rational place, we know that $q,q+1 \in H(P)$. Let us recursively define functions $h_j \in L((j+1)qP_\infty)$ as follows: $h_0:=x_a-y_b$, $h_1:=-x_a+y_b+(x_a+y_b)^2$ and for $j \ge 2$, $h_j:=(y_b^2-x_a^2)h_{j-2}-h_{j-1}.$ Since $x_a,y_b \in L(qP_\infty)$ by Equations \eqref{divx1} and \eqref{divy1}, it follows directly from the definition that $h_j \in L((j+1)qP_\infty).$

Using the same notation $T$ as before, we claim that for all $j \ge 0$ one has the following power series developments at $P$:
\begin{equation}\label{eq:gj_powerseries}
h_j=T^{3j+2}+T^{3j+3}+O(T^q).
\end{equation}
Since $\beta(P)=1$, Equations \eqref{eq:xa_powerseries} and \eqref{eq:yb_powerseries} imply Equation \eqref{eq:gj_powerseries} holds for $j=0$ and $j=1$. Now assume that $j \ge 2$ and that Equation \eqref{eq:gj_powerseries} holds for $j-2$ and $j-1$. Then from the recursive definition of $h_j$ one finds
\begin{eqnarray*}
h_j & = & (T^3+T^6)(T^{3j-4}+T^{3j-3})-(T^{3j-1}+T^{3j})+O(T^q)\\
& = & T^{3j+2}+T^{3j+3}+O(T^q),
\end{eqnarray*}
which is what we needed to show.

Now consider the function $H_j:=h_j/F_P^{j+1}$, then $v_{P_\infty}(H_j) \ge -(j+1)q+(j+1)q=0$, which implies that $H_j$ only can have a pole at $P$. Further, Equation \eqref{eq:gj_powerseries} implies that for all $j$ satisfying $0 \le j \le m-1$ one has 
$$-v_P(H_j)=-(3j+2)+(j+1)(q+1)=(q-1)+j(q-2).$$
We may conclude that 
\begin{equation*}
\langle q, q+1, (q-1) +j(q-2) \mid j=0,\ldots, m-1\rangle \subseteq H(P).
\end{equation*}

What is left is to show that the semigroup $H:=\langle q, q+1, (q-1) +j(q-2) \mid j=0,\ldots, m-1\rangle$ has at most $g(Z_3)$ many gaps. Using the generators of $H$, it is clear that the only possible gaps of $H$ are contained in one of the sets 
$$\{1,\dots,q-2\}, \{q+2,\dots,2q-4\}, \{2q+3,\dots,3q-6\}\dots,\{m(q-2)\}.$$ Hence one obtains that the number of gaps of $H$ is at most
\[(q-2)+(q-5)+\cdots+1=\sum_{k=0}^{m-1}(q-2-3k)=m(q-2)-3\frac{m(m-1)}{2}=\frac{q^2-q}6.\]
\end{proof}
Note that the proof of Theorem \ref{sembeta1} has similarities with the proof of Theorem 5.4 in \cite{BMV2}.
Because of the results in this subsection, we are left with considering the case $\beta(P)^{(q-1)/2}=-1$ in order to obtain the complete determination of the structure of $H(P)$ if $P$ is a rational place.
In the next section, we recursively construct some functions that will be fundamental in determining $H(P)$ for these places as well as for the non-rational places. 

\subsection{Three classes of special polynomials}

We use the same notation as in the previous section: for a place $P=P_{(a,b)}$, we define $\beta(P):=(b+b^3+\cdots+b^{q/3})^2=-(a^q+a)$. In this subsection, we often just write $\beta$ instead of $\beta(P)$ for simplicity. Further, if $P_{(a,b)}$ lies below a point $Q_{(A,B)}$ of the Hermitian function field, then just as before we write $T=(v-B)/(B^q-B).$ Here we assume that $B^q-B=p(b)\neq 0$. Again as before, we write $x_a=-(x-a)/\beta$ and $y_b=-(y-b)/(B^q-B)$, and Equations \eqref{eq:xa_powerseries} and \eqref{eq:yb_powerseries} state that
\begin{eqnarray*}
x_a  = T+T^2+O(T^q) \quad \text{and} \quad y_b =  T-\beta T^3+O(T^q).
\end{eqnarray*}

We will construct two families of function $f_j \in L((j+1)qP_\infty)$ and $g_\ell \in L((3\ell+4)m P_\infty)$ vanishing in $P$ with various multiplicities. In order to define and investigate these functions and more in particular the coefficients appearing in their power series in $T$, we will need three families of polynomials. In this subsection, we define and study these polynomials. In the next subsection, we construct the functions $f_j$ and $g_\ell$.

\begin{definition}\label{def:PQR}
We define the polynomials $\mathcal{P}_{1}(s)=1$, $\mathcal{Q}_{1}(s)=s$, $\mathcal{R}_1(s)=-(s+1)$,  $\mathcal{P}_{2}(s)=-s^3$, $\mathcal{Q}_{2}(s)=s^4-s^3-s^2$ and $\mathcal{R}_2(s)=-(s^4-s^3+s).$ Moreover, for any integer $j\ge 2$, we recursively define
polynomials $\mathcal{P}_{j+1}(s),\mathcal{Q}_{j+1}(s),\mathcal{R}_{j+1}(s)$ as follows:
\begin{equation}\label{eq:recdefP}
\mathcal{P}_{j+1}(s)=\mathcal{P}_{2}(s) \cdot \mathcal{P}_{j}(s)-(s(s-1))^3\cdot\mathcal{P}_{j-1}(s), 
\end{equation}
\begin{equation}\label{eq:recdefQ}
\mathcal{Q}_{j+1}(s)={\mathcal P}_2(s) \cdot \mathcal{Q}_{j}(s)-(s(s-1))^3\cdot\mathcal{Q}_{j-1}(s) \end{equation}
and
\begin{equation*}
\mathcal{R}_{j+1}(s)={\mathcal P}_2(s) \cdot \mathcal{R}_{j}(s)-(s(s-1))^3\cdot\mathcal{R}_{j-1}(s) 
\end{equation*}
\end{definition}

It will also be convenient to define ${\mathcal P}_0(s)=0$, ${\mathcal Q}_0(s)=(s(s-1))^{-1}$ and ${\mathcal R}_0(s)=-1/s^2$. The effect of this is that the recursive equations in Definition \ref{def:PQR} also hold for $j=1$.

\begin{remark}
Note that
\[\left[
\begin{array}{c}
\mathcal{R}_0(s)\\
\mathcal{R}_1(s)
\end{array}\right]=\left[
\begin{array}{c}
\mathcal{P}_0(s)\\
\mathcal{P}_1(s)
\end{array}\right]-\frac{s-1}{s}\left[
\begin{array}{c}
\mathcal{Q}_0(s)\\
\mathcal{Q}_1(s)
\end{array}\right].
\]
Since the polynomials $\mathcal{R}_j(s)$ satisfy the same second order linear recursion as the polynomials $\mathcal{P}_j(s)$ and $\mathcal{Q}_j(s)$, this implies that
\begin{equation}\label{eq:RinPQ}
\mathcal{R}_j(s)=\mathcal{P}_j(s)-\frac{s-1}{s}\mathcal{Q}_j(s) \quad \text{for all $j \ge 0$}.
\end{equation}
\end{remark}

Before returning to constructing the functions $f_j$ and $g_\ell$ above, we need to collect several facts about the polynomials we just defined.

\begin{lemma}
    \label{lem:PQ:identities}
Let $i,j,\ell \in \Z_{\ge 0}$. Then
\begin{equation} \label{eq:id:1}
    \PP_{i+j}(s)\PP_{i+\ell}(s)-\PP_i(s) \PP_{i+j+\ell}(s) =(s^6-s^3)^{i} \PP_{j}(s) \PP_{\ell}(s),
\end{equation}
\begin{equation} \label{eq:id:2}
    \PP_{i+j}(s)\QQ_{i+\ell}(s)-\PP_i(s) \QQ_{i+j+\ell}(s) =(s^6-s^3)^{i} \PP_{j}(s) \QQ_{\ell}(s).
\end{equation}
and
\begin{equation}
    \label{eq:id:3}
    \PP_{i+j}(s)\RR_{i+\ell}(s)-\PP_i(s) \RR_{i+j+\ell}(s) =(s^6-s^3)^{i} \PP_{j}(s) \RR_{\ell}(s).
\end{equation}
\end{lemma}

\begin{proof}
We prove the second identity only, since the proof of the first and last ones are very similar. For convenience, we leave out the variable for polynomials such as $\PP_i(s)$, $\QQ_{i}(s)$ and simply write $\PP_i$ and $\QQ_i$.

We prove the identity using induction on $i$. If $i=0$, Equation \eqref{eq:id:2} is obvious, using that $\PP_0=0.$ If $i=1$, induction on $j$ can be used.

Now assume that $i \ge 1$ and that Equation \eqref{eq:id:2} holds for $i$ and $i-1$ and all $j$ and $\ell$. Then using the recursive definition from Equations \eqref{eq:recdefP} and \eqref{eq:recdefQ} together with the induction hypothesis, we find:
\begin{eqnarray*}
\PP_{i+j+1} \QQ_{i+\ell+1} -\PP_{i+1} \QQ_{i+j+\ell+1} & = & 
(\PP_2\PP_{i+j}-(s^6-s^3)\PP_{i+j-1})\QQ_{i+\ell+1}-(\PP_2 \PP_i-(s^6-s^3)\PP_{i-1})\QQ_{i+j+\ell+1}\\
& = & \PP_2 \cdot (\PP_{i+j}\QQ_{i+\ell+1}-\PP_i\QQ_{i+j+\ell+1})-(s^6-s^3)(\PP_{i+j-1}\QQ_{i+\ell+1}-\PP_{i-1}\QQ_{i+j+\ell+1})\\
& = & \PP_2 \cdot (s^6-s^3)^{i} \PP_{j} \QQ_{\ell+1}-(s^6-s^3)^{i} \PP_{j} \QQ_{\ell+2}\\
& = & (s^6-s^3)^{i} \PP_{j} \cdot (\PP_2 \QQ_{\ell+1} -\QQ_{\ell+2})\\
& = & (s^6-s^3)^{i+1} \PP_{j} \QQ_\ell.
\end{eqnarray*}
\end{proof}

\begin{remark}\label{rem:common_roots_P_Q}
Choosing $j=1$ and $\ell=0$ in Equation \eqref{eq:id:2} one obtains that
$$\PP_{i+1}(s)\QQ_{i}(s)-\PP_i(s) \QQ_{i+1}(s) =(s^6-s^3)^{i} \cdot \frac{1}{s(s-1)}=s^{3i-1}(s-1)^{3i-1}.$$
This implies that the only possible common roots of $\PP_{i}(s)$ and $\QQ_i(s)$ are $0$ and $1$.
Similarly, one can use Equation \eqref{eq:id:3} to show that $0$ and $1$ are the only possible common roots of $\PP_{i}(s)$ and $\RR_i(s)$.
\end{remark}

Next, we derive a closed formula for the polynomials $\PP_i(s)$, $\QQ_i(s)$ and $\RR_i(s)$. 

\begin{theorem}\label{thm:PQformulas}
Let $i \in \mathbb{Z}_{\ge 0}$. Then we have
\[\PP_i(s)=\frac{-(s^3+s \sqrt{s})^i+(s^3-s \sqrt{s})^i}{s\sqrt{s}},\]
\[\QQ_i(s)=\frac{(\sqrt{s}-1)(s^3+s \sqrt{s})^i-(\sqrt{s}+1)(s^3-s \sqrt{s})^i}{s(s-1)}\]
and
\[\RR_i(s)=\frac{(\sqrt{s}+1)(s^3+s \sqrt{s})^i-(\sqrt{s}-1)(s^3-s \sqrt{s})^i}{s^2}.\]
\end{theorem}

\begin{proof}
Using Equation \ref{eq:RinPQ}, the third identity follows directly from the first two. We therefore only need to prove the first two identities. 
One can verify directly that these are true for $i=0$ and $i=1$.
Equations \eqref{eq:recdefP} and \eqref{eq:recdefQ} imply that for $i \ge 2$
\[
\left[\begin{array}{cc}
 \PP_2(s)    & -s^6+s^3 \\
   1  &  0
   \end{array}\right]
\cdot
\left[\begin{array}{cc}
 \PP_i(s)    & \QQ_i(s) \\
 \PP_{i-1}(s)   & \QQ_{i-1}(s)
\end{array}\right]
=
\left[\begin{array}{cc}
 \PP_{i+1}(s)    & \QQ_{i+1}(s) \\
 \PP_{i}(s)   & \QQ_{i}(s)
\end{array}\right].
\] 
Hence for $i \ge 2$, we have
\[
\left[\begin{array}{cc}
 \PP_i(s)    & \QQ_i(s) \\
 \PP_{i-1}(s)   & \QQ_{i-1}(s)
\end{array}\right]
=
\left[\begin{array}{cc}
 \PP_2(s)    & -s^6+s^3 \\
   1  &  0
   \end{array}\right]^{i-2}
\cdot
\left[\begin{array}{cc}
 \PP_{2}(s)    & \QQ_{2}(s) \\
 \PP_{1}(s)   & \QQ_{1}(s)
\end{array}\right].
\] 
Note that the matrix $\left[\begin{array}{cc}
 \PP_2(s)    & -s^6+s^3 \\
   1  &  0
\end{array}\right]$ has eigenvalues $s^3 \pm s\sqrt{s}$ with corresponding eigenvectors for example 
$$
\left[\begin{array}{c}
 s^3 \pm s \sqrt{s} \\
   1 
\end{array}\right].$$
Hence for $i \ge 2$, we have
\begin{multline*}
\left[\begin{array}{cc}
 \PP_i(s)    & \QQ_i(s) \\
 \PP_{i-1}(s)   & \QQ_{i-1}(s)
\end{array}\right]
=
\left[\begin{array}{cc}
 s^3+s\sqrt{s}    & s^3-s\sqrt{s} \\
   1  &  1
   \end{array}\right]
\cdot
\left[\begin{array}{cc}
 (s^3+s\sqrt{s})^{i-2}    & 0 \\
   0  &  (s^3-s\sqrt{s})^{i-2}
   \end{array}\right]
\cdot \\
\left[\begin{array}{cc}
 s^3+s\sqrt{s}    & s^3-s\sqrt{s} \\
   1  &  1
   \end{array}\right]^{-1}
   \cdot
\left[\begin{array}{cc}
 \PP_{2}(s)    & \QQ_{2}(s) \\
 \PP_{1}(s)   & \QQ_{1}(s)
\end{array}\right].
\end{multline*}
Working out the matrix products on the right-hand side of this equation and simplifying the outcome, the desired identities follow for $i \ge 2$.
\end{proof}

\begin{corollary}\label{cor:originaldefinitionR}
Let $i \in \Z_{\ge 1}$, then
\[\RR_i(s)=\RR_{i-1}(s) s (s-1)^2 + \PP_i(s)/s.\]
\end{corollary}
\begin{proof}
This follows directly using the explicit formulas given in Theorem \ref{thm:PQformulas}.
\end{proof}

Since, as we will see later, the polynomials from Definition \ref{def:PQR} occur as coefficients in power series expansions of function we wish to construct, we are also interested in determining if these coefficients are zero or not. This motivates the following definition.

\begin{definition}
\label{def:Pord}
Let $\beta \in \overline{\F}_{q^2} \setminus \{0,1\}$. Then we define the $\PP$-order (resp. $\RR$-order) of $\beta$ as the smallest positive integer $i$ such that $\PP_{i+1}(\beta)=0$ (resp. $\RR_{i+1}(\beta)=0$). If $P$ is a place of $\Zte$ such that $\beta(P) \not \in \{0,1,\infty\}$, then define the $\PP$-order (resp. $\RR$-order) of $P$ to be the $\PP$-order (resp. $\RR$-order) of $\beta(P)$.
\end{definition}

Note that this definition plays a similar role as Definition 3.10 from \cite{BMV}. We finish this subsection with some remarks.

\begin{remark}\label{rem:concerningPorder}
Note that Theorem \ref{thm:PQformulas} implies that $P$ with $\beta(P) \not\in \{0,1,\infty\}$ has $\mathcal{P}$-order $i$ if and only if $i$ is the smallest positive integer such that the element $\gamma:=(\sqrt{\beta(P)}+1)/(\sqrt{\beta(P)}-1)$ satisfies $\gamma^{3i+3}=1$. Since the characteristic is three, this is equivalent with $\gamma^{i+1}=1$. Therefore the $\PP$-order of $\beta$ is $i$ if and only if the corresponding element $\gamma$ has multiplicative order $i+1$. In particular, the $\PP$-order is always a finite number. Note that $i \ge 2$. Indeed, $i=0$ implies $\gamma=1$, which would imply $1=-1$, while $i=1$ implies $\gamma=-1$, which would imply that $\sqrt{\beta(P)}=0$. 
\end{remark}

\begin{remark}\label{rem:reltionPandRorder}
Similarly as in the previous remark, the explicit formula for $\RR_{K+1}(s)$ implies that $P$ with $\beta(P) \not\in \{0,1,\infty\}$, has finite $\RR$-order $K$ if and only if $K$ is the smallest positive integer such that $\gamma^{3K}=\gamma^{-1}$, that is to say, such that $\gamma^{3K+1}=1$. Note that such a $K$ always exists. Indeed if $\gamma$ has multiplicative order $j$, then $j$ is not a multiple of three, since, using that the characteristic is three, if it were, then $\gamma^j=1$ would imply $\gamma^{j/3}=1$. Hence either $j$ or $2j$ is one modulo three. This means that $K=(j-1)/3$ if $j$ is one modulo three and $K=(2j-1)/3$ if $j$ is two modulo three. In view of the previous remark, we see that if $i$ is the $\PP$-order of $P$, then its $\RR$-order $K$ equals $K=i/3$ if $i$ is zero modulo three and $K=(2i+1)/3$ if $i$ is one modulo three. Note that this implies $K<i$, since $i \ge 2.$ 
\end{remark}

\begin{remark}\label{rem:P-order_rational}
If $P$ is a rational place of $Z_3$ with $\beta(P) \not\in \{0,1,\infty\}$, then according to Lemma \ref{lem:beta_rational} we have $\beta(P)^{(q-1)/2}=-1$. Hence $\sqrt{\beta}^{q-1}=-1$, which implies that $\sqrt{\beta}^{q}=-\sqrt{\beta}$. As a consequence, if we write as in Remark \ref{rem:concerningPorder} $\gamma=(\sqrt{\beta(P)}+1)/(\sqrt{\beta(P)}-1)$, then $\gamma^q=\gamma^{-1}$. This means that the multiplicative order of $\gamma$ divides $q+1$. In particular, $i$ is the $\PP$-order of a rational place $P$ with $\beta(P) \not\in \{0,1,\infty\}$ only if $i+1$ divides $q+1$ and $i \ge 2$. Conversely, it is not hard to see that if $i \ge 2$ is a positive integer such that $i+1$ divides $q+1$, then there exist exactly $q^2/3\cdot \varphi(i+1)$ rational places $P$ with $\PP$-order $i$. Indeed, there are precisely $\varphi(i+1)$ rational places $Q$ of $\mathbb{F}_{q^2}(\sqrt{\beta})=\mathbb{F}_{q^2}(p(y))$ such that $\gamma=(\sqrt{\beta(Q)}+1)/(\sqrt{\beta(Q)}-1) \in \mathbb{F}_{q^2}$ has multiplicative order $i+1$ and all these places split completely in the extension $Z_3/\mathbb{F}_{q^2}(\sqrt{\beta})$, which is a Galois extension of degree $q^2/3$.
\end{remark}

\subsection{Two classes of special functions}

We now construct the functions $f_j$ and $g_\ell$ alluded to in the previous subsection.
We start by considering the function $f_0:=x_a-y_b$. One has 
\begin{equation} \label{f0}
f_0=T^2+\beta T^3+O(T^q).
\end{equation}
Now define $f_1:=f_0-x_a^2-(\beta+1)x_a f_0+(\beta^2-\beta-1)f_0^2$, then 
\begin{equation} \label{f1}
f_1=2\beta^3 T^5+(\beta^4-\beta^3-\beta^2)T^6+O(T^q).
\end{equation}
Similarly, defining $f_2:=(\beta^2-\beta)f_0f_1+\beta^3f_0^2x_a+f_1+\beta^2f_0^3$, one obtains
\begin{equation} \label{f2}
f_2=\beta^3 T^8+(\beta^7+\beta^6+\beta^5+\beta^4)T^9+O(T^q).
\end{equation}
We are going to construct functions $f_j \in L((j+1)qP_\infty)$ such that their power series at $P$ satisfy
\begin{equation}\label{eq:fj_powerseries1}
f_j=\mathcal{P}_{j+1}(\beta)T^{3j+2}+\mathcal{Q}_{j+1}(\beta)T^{3j+3}+O(T^q),
\end{equation}
where $\mathcal{P}_{j+1}(s),\mathcal{Q}_{j+1}(s)$ are the polynomials from Definition \ref{def:PQR}. 

Note that the functions $f_0$ and $f_1$ satisfy Equation \eqref{eq:fj_powerseries1} using Equations \eqref{f0} and \eqref{f1} and the definitions of $\mathcal{P}_j(s)$ and $\mathcal{Q}_j(s)$ for $j=1,2$.
The function $f_2$ we constructed also satisfies Equation \eqref{eq:fj_powerseries1}. Indeed this follows using Equation \eqref{f2} and the fact that
$$\mathcal{P}_{3}(s)=\mathcal{P}_{2}(s)^2-s^3(s^3-1)\mathcal{P}_{1}(s)=s^3$$ and $$\mathcal{Q}_{3}(s)=\mathcal{P}_{2}(s)\mathcal{Q}_{2}(s)-s^3(s^3-1)\mathcal{Q}_{1}(s)=s^7+s^6+s^5+s^4.$$ 
Next we want to define $f_j \in L((j+1)qP_\infty)$ for $j \ge 3$. Note that since $x_a,y_b \in L(qP_\infty)$, it immediately follows that $f_j \in L((j+1)qP_\infty)$ for $j=0,1,2$.

\begin{definition}
\label{def:fj}
Let $P$ be a place of $\Zte$ such that $\beta(P) \not \in \{0,1,\infty\}$. Suppose that $P$ has $\mathcal{P}$-order $i$. Then for $ 3 \le j \le i$, we recursively define $$f_j=\frac{\mathcal{P}_2(\beta(P))\mathcal{P}_{j-1}(\beta(P))f_0 f_{j-1}-\mathcal{P}_{j}(\beta(P))f_1 f_{j-2}}{\beta(P)^2 (\beta(P)-1)^2 \mathcal{P}_{j-2}(\beta(P))}.$$
\end{definition}

Note that it is easy to show that $f_j \in L((j+1)qP_\infty)$ by the recursive definition and the fact that $f_j \in L((j+1)qP_\infty)$ for $j=0,1,2$. Also note that since $3 \le j \le i$, the term $\mathcal{P}_{j-2}(\beta(P))$ in the numerator in Definition \ref{def:fj} is not zero. As we already showed that Equation \eqref{eq:fj_powerseries1} holds for $j=0,1,2$, we now show it holds for $j \ge 3$ as well: 

\begin{theorem}
\label{thm:fi}
Let $P$ be a place of $\Zte$ such that $\beta(P)\not \in \{0,1,\infty\}$, and let $i$ be the $\PP$-order of $\beta(P)$. If $i \le m-1$, then the function $f_i$ defined in \cref{def:fj} is such that $v_{P}(f_{i})=3i+3$. On the other hand, for each $j\in \mathbb{Z}$ with $0\leq j\leq \min\{i-1,m-1\}$, the function $f_j$ from \cref{def:fj} is such that $v_{P}(f_j)=3j+2$.
\end{theorem}

\begin{proof}
As before, we consider a place $Q_{(A,B)}$ of the Hermitian function field lying over $P:=P_{(a,b)}$, and we define $T:=(v-B)/(B^q-B)$, a local parameter at $Q_{(A,B)}$. Note that $e(Q_{(A,B)}|P)=1$, since $P \neq P_\infty$.
Also, as in the proof of \cref{lem:PQ:identities}, for simplicity we just write $\beta$ instead of $\beta(P)$ and $\PP_j$, $\QQ_j$ instead of $\PP_j(\beta(P))$ and $\QQ_j(\beta(P))$. 

For each $j$ such that $0\leq j\leq i$, we wish to show Equation \eqref {eq:fj_powerseries1}.
Note that this is sufficient to prove the theorem. The second part follows since, for all $j\in \mathbb{Z}$ such that $0\leq j\leq \min\{i-1,m-1\}$, it holds that $3j+3 <q$ and $v_{Q_{(A,B)}}(f_j)=v_{P_{(a,b)}}(f_j)$. The first part follows since $3i+3 <q$ and, by definition of the $\PP$-order, Equation \eqref{eq:fj_powerseries1} will give that
\begin{equation*}
f_i = \QQ_{i+1}  T^{3i+3} + O(T^{q}).
\end{equation*}
Since, by \cref{rem:common_roots_P_Q}, we know $\mathcal{Q}_{i+1} \neq 0$, the result follows because $v_P(f_i)=v_{Q_{(A,B)}}(f_i)=3i+3$.

Hence, by \cref{def:fj}, we need to show that
\begin{equation*}
   \PP_2\PP_{j-1} f_{j-1}f_0 - \PP_j f_{j-2}f_1 = \beta^2(\beta-1)^2 \PP_{j-2}\left(\mathcal{P}_{j+1} T^{3j+2} + \QQ_{j+1}  T^{3j+3} + O(T^{q})\right).
\end{equation*}
The local power series expansion of $\PP_2\PP_{j-1} f_{j-1}f_0 - \PP_j f_{j-2}f_1$ with respect to $T$ can be obtained from the expansions of the functions $f_{j-1}f_0$ and $f_{j-2}f_1$, which are:
\begin{equation*}
    \begin{aligned}
    f_{j-2}f_1 & = \left(\PP_{j-1}T^{3j-4} + \QQ_{j-1}  T^{3j-3} + O(T^{q})\right)\left(\PP_2 T^5+\QQ_2 T^6+O(T^q)\right)\\
    & = \PP_{2} \PP_{j-1} T^{3j+1} + (\PP_{j-1}\QQ_{2}  + \PP_2\QQ_{j-1})T^{3j+2} + \QQ_2 \QQ_{j-1} T^{3j+3} + O(T^{q}),\\
    f_{j-1}f_0 & = \left(\PP_{j}T^{3j-1} + \QQ_{j}  T^{3j} + O(T^{q})\right)\left(T^{2} + \QQ_1 T^{3} + O(T^{q})\right)\\
    & = \PP_{j}T^{3j+1} + (\PP_{j}\QQ_{1}  + \QQ_{j})T^{3j+2} + \QQ_1 \QQ_{j} T^{3j+3} + O(T^{q}).
    \end{aligned}
\end{equation*}
Hence, we have
\begin{equation*}
\begin{split}
    \PP_2\PP_{j-1} f_{j-1}f_0 - \PP_j f_{j-2}f_1 = (\PP_2 \PP_{j-1} \PP_j \QQ_1  + \PP_2 \PP_{j-1} \QQ_{j} - \PP_{j-1} \PP_j \QQ_2  - \PP_2 \PP_j \QQ_{j-1})T^{3j+2}\\
     + \left(\PP_{j-1} \PP_2 \QQ_{j}  \QQ_1 - \PP_j \QQ_{j-1}  \QQ_2   \right)T^{3j+3} + O(T^{q}).
\end{split}
\end{equation*}
We are therefore left to prove the two following identities:
\begin{equation}
\label{eq:coeff:T3i+2}
    \PP_2 \PP_{j-1} \PP_j \QQ_1  + \PP_2 \PP_{j-1} \QQ_{j} - \PP_{j-1} \PP_j \QQ_2  - \PP_2 \PP_j \QQ_{j-1}
    = \beta^2(\beta-1)^2 \PP_{j-2} \PP_{j+1}
\end{equation}
and
\begin{equation}
\label{eq:coeff:T3i+3}
    \PP_{j-1} \PP_2 \QQ_{j}  \QQ_1 - \PP_j \QQ_{j-1}  \QQ_2  = \beta^2(\beta-1)^2 \PP_{j-2} \QQ_{j+1} .
\end{equation}
This can be conveniently done by using Lemma \ref{lem:PQ:identities}. Indeed, consider first Equation \eqref{eq:coeff:T3i+2} and use identity \eqref{eq:id:2} as
\begin{equation*}
    \PP_j \QQ_{2} -\PP_2 \QQ_{j} =\PP_{j-2} \QQ_{0} \cdot (\beta^6-\beta^3)^{2}=\PP_{j-2} \QQ_{0} \cdot (\beta^2-\beta)^{6}=\PP_{j-2} \cdot (\beta^2-\beta)^{5},
\end{equation*}
i.e., with indices $(2,j-2,0)$ (listed in order as in the statement of Lemma \ref{lem:PQ:identities}).
Then, we obtain
\begin{equation}
\label{eq:first:part:1}
\begin{split}
    \PP_{j-1} \PP_j \QQ_2  - \PP_2 \PP_{j-1} \QQ_{j} &= \PP_{j-1}(\PP_j\QQ_2  - \PP_2\QQ_{j} )\\
    &= \PP_{j-1}\cdot(\beta^2 - \beta)^5 \PP_{j-2}.
\end{split}
\end{equation}
By using again Equation \eqref{eq:id:2}, this time with indices $(1,j-2,0)$,
we can also rewrite
\begin{equation}
\label{eq:second:part:1}
\begin{split}
    \PP_2 \PP_j \QQ_{j-1}  - \PP_2 \PP_{j-1} \PP_j \QQ_1  &= \PP_2 \PP_j \QQ_{j-1}  \PP_1 - \PP_2 \PP_{j-1} \PP_j \QQ_1 \\
    &= -\PP_2\PP_j (\PP_{j-1}\QQ_1  - \PP_1\QQ_{j-1} )\\
    &= -\PP_2\PP_j \cdot (\beta^2 - \beta)^2 \PP_{j-2}.
\end{split}
\end{equation}
Then, by Equations \eqref{eq:first:part:1} and \eqref{eq:second:part:1}, we have that Equation \eqref{eq:coeff:T3i+2} is equivalent to
\begin{equation*}
    -\PP_{j-1}\cdot (\beta^2 - \beta)^5 \PP_{j-2}+\PP_2\PP_j \cdot (\beta^2 - \beta)^2 \PP_{j-2} = (\beta^2 - \beta)^2 \PP_{j-2} \PP_{j+1}.
\end{equation*}
Dividing out the factor $(\beta^2 - \beta)^2 \PP_{j-2}$ both in the right hand side and the left hand side of this equality and rearranging the terms, we obtain
\begin{equation*}
    \PP_{j-1}\cdot (\beta^2 - \beta)^3 = \PP_2 \PP_j -\PP_{j+1},
\end{equation*}
which holds by Lemma \ref{lem:PQ:identities}, as it is precisely identity \eqref{eq:id:1} with indices $(1,j-1,1)$.

In order to prove Equation \eqref{eq:coeff:T3i+3}, we can argue in a similar way. Indeed, we have:
\begin{multline*}
    \PP_j \QQ_{j-1}  \QQ_2  - \PP_{j-1} \PP_2 \QQ_{j}  \QQ_1  =
    (\PP_j \QQ_2  - \PP_2 \QQ_{j}  + \PP_2 \QQ_{j} )\QQ_{j-1}  - \PP_{j-1} \PP_2 \QQ_{j}  \QQ_1 \\
    =\left(\PP_{j-2}\cdot (\beta^2 - \beta)^5 + \PP_2\QQ_{j} \right)\QQ_{j-1}  - \PP_{j-1} \PP_2 \QQ_{j}  \QQ_1,
\end{multline*}
where the last equality follows from Equation \eqref{eq:id:2} with indices $(2,j-2,0)$.
Moreover,
\begin{multline*}
    \left(\PP_{j-2}\cdot (\beta^2 - \beta)^5 + \PP_2\QQ_{j} \right)\QQ_{j-1}  - \PP_{j-1} \PP_2 \QQ_{j}  \QQ_1  = \\
    \PP_{j-2}\cdot (\beta^2 - \beta)^5\QQ_{j-1}  - \PP_2\QQ_{j} \left(\PP_{j-1}\QQ_1  - \PP_1 \QQ_{j-1} \right) = \\
    \PP_{j-2}\cdot (\beta^2 - \beta)^5\QQ_{j-1}  - \PP_2 \QQ_{j} \PP_{j-2}\cdot (\beta^2 - \beta)^2,
\end{multline*}
where the last equality follows from Equation \eqref{eq:id:2} with indices $(1,j-2,0)$.
Finally, using again Equation \eqref{eq:id:2} with indices $(1,1,j-1)$,
we have
\begin{multline*}
    \PP_{j-2}\cdot(\beta^2 - \beta)^5\QQ_{j-1}  - \PP_2 \QQ_{j} \PP_{j-2}\cdot(\beta^2 - \beta)^2 = \\
    -\PP_{j-2}\cdot(\beta^2 - \beta)^2\left(\PP_2 \QQ_{j}  - \QQ_{j-1}\cdot (\beta^2 - \beta)^3 \PP_1\right) =
     -\PP_{j-2}\cdot(\beta^2 - \beta)^2 \QQ_{j+1},
\end{multline*}
which proves Equation \eqref{eq:coeff:T3i+3}.
\end{proof}

To deal with computing $G(P)$ (equivalently $H(P)$) in the case in which $P$ is not $\mathbb{F}_{q^2}$-rational, another family of special recursive functions is going to be needed. A similar phenomenon was already visible in \cite{BMV2}, and in fact the proposition that follows is a natural analogue for $Z_3$ of \cite[Proposition 6.9]{BMV2}. 

\begin{proposition}
\label{prop:functions:gell}
Let $P$ be a place of $\Zte$ such that $\beta(P) \not \in  \{0,1,\infty\}$. Further, let $K$ be the $\RR$-order of $P$. If $K \le m-2$, then there exists a function $g_{K}\in L((3K + 4)mP_\infty)$ such that $v_{P}(g_{K})=3K+4$. Moreover, for each $\ell\in \mathbb{Z}$ with $0\leq \ell \leq \min\{K-1,m-2\}$, there exists a function $g_\ell\in L((3\ell + 4)mP_\infty)$ with $v_{P}(g_\ell)=3\ell+3$.
\end{proposition}

\begin{proof}
In the proof, we simplify the notations by writing $\beta$ instead of $\beta(P)$ and $\RR_{\ell}$, $\PP_\ell$ instead of $\RR_{\ell}(\beta(P))$ and $\PP_\ell(\beta(P))$.

Consider the following recursively defined family of functions in $Z_3$:
\begin{equation*}
    g_0:= x_a^2 - f_0,
\end{equation*}
and, for $\ell \geq 1$,
\begin{equation}
\label{eq:gell}
    g_\ell:= \frac{\PP_{\ell+1} \cdot g_{\ell-1} \cdot f_0 - \RR_\ell\cdot f_\ell}{\PP_\ell \cdot \beta},
\end{equation}
where the functions $f_\ell$ are those from \cref{def:fj}. Note that the term $\PP_\ell$ in the numerator is not zero, since $1 \le \ell \le K$ and $K$ is strictly less than the $\PP$-order of the place $P$ by Remark \ref{rem:reltionPandRorder}. 
Also note that if $K\leq m-2$, then $3K+4 \leq 3(m-2)+4=q-2$ and similarly that if $\ell \le \min\{K-1,m-2\}\leq m - 2$, then $3\ell+4 \leq 3(m-2)+4=q-2$.

Our first claim is that for all $0 \leq \ell \leq \min\{K,m-2\}$, $g_\ell \in L((3\ell + 4)mP_\infty)$ and that the local power series expansion of $g_\ell$ at $P$ is:
\begin{equation}
\label{eq:powerseriesgell}
    g_\ell = \RR_{\ell+1}T^{3\ell + 3} + \PP_{\ell+1}T^{3\ell + 4} + O(T^q).
\end{equation}
We prove the claim by induction on $\ell$.

We start observing that, by construction, it immediately follows that
\begin{equation*}
    g_0 = -(\beta + 1)T^3 + T^4 + O(T^q) = \RR_1 T^3 + \PP_1T^4 + O(T^q),
\end{equation*}
and that $g_0\in L(4mP_\infty)$, since $f_0\in L(3mP_\infty)$ and $x_a^2\in L(4mP_\infty)$.

For $\ell \geq 1$, we now consider the functions defined in Equation \eqref{eq:gell}. Using the induction hypothesis for $\ell-1$, we have
\begin{align*}
    \frac{\PP_{\ell+1} \cdot g_{\ell-1} \cdot f_0 - \RR_\ell\cdot f_\ell}{\PP_\ell \cdot \beta} &= 
    \frac{(\PP_{\ell+1}\RR_\ell - \RR_\ell \PP_{\ell+1})T^{3\ell + 2} + (\PP_{\ell+1}\RR_\ell\QQ_1 + \PP_{\ell+1}\PP_\ell - \RR_\ell \QQ_{\ell+1})T^{3\ell+3}}{\PP_\ell \cdot \beta}\\
    &+ \frac{(\PP_{\ell+1} \PP_\ell \QQ_1)T^{3\ell+4} + O(T^q)}{\PP_\ell \cdot \beta}\\
     &= \frac{(\RR_\ell\beta^2 (\beta-1)^2\PP_\ell + \PP_{\ell+1}\PP_\ell)T^{3\ell+3} + (\PP_{\ell+1} \PP_\ell \QQ_1)T^{3\ell+4} + O(T^q)}{\PP_\ell \cdot \beta}\\
     &=(\RR_\ell\beta (\beta-1)^2 + \PP_{\ell+1}/\beta)T^{3\ell+3} + \PP_{\ell+1}T^{3\ell+4} + O(T^q)\\
     &=\RR_{\ell +1} T^{3\ell+3} + \PP_{\ell+1}T^{3\ell+4} + O(T^q),
\end{align*}
where the third-to-last equality above follows since, from Equation \eqref{eq:id:2}, we have
\begin{equation*}
    \PP_{\ell+1}\RR_\ell\QQ_1 + \PP_{\ell+1}\PP_\ell - \RR_\ell \QQ_{\ell+1} = \RR_\ell(\PP_{\ell+1}\QQ_1 - \QQ_{\ell+1}\PP_1) + \PP_{\ell+1}\PP_\ell = \RR_\ell\beta^2 (\beta-1)^2\PP_\ell + \PP_{\ell+1}\PP_\ell,
\end{equation*}
the second-to-last equality follows from the fact that $\QQ_1=\beta$, and the last equality follows from Corollary \ref{cor:originaldefinitionR}.
To conclude, we observe that $f_\ell\in L((3\ell+3)mP_\infty)$, while, by induction hypothesis, $g_{\ell-1}f_0\in L((3\ell+4)mP_\infty)$, and hence $g_\ell\in L((3\ell+4)mP_\infty)$.

We now observe that, if $K\leq m-2$ is the $\RR$-order of $P$, then $\RR_{K+1}=0$ and $\PP_{K+1} \neq 0$ using the final part of \cref{rem:common_roots_P_Q}. This shows that $v_P(g_K) = 3K + 4$. On the other hand, it is also clear from Equation \eqref{eq:powerseriesgell} that, for all $0\leq \ell \leq \min\{K-1,m-2\}$, $v_P(g_\ell)=3\ell+3$.
\end{proof}

\section{Weierstrass semigroup at the remaining places of $\Zte$}

The aim of this section is to compute either the Weierstrass semigroup $H(P)$ or its set of gaps $G(P)$ for all the remaining places $P$ of $\Zte$ not considered above, that is, where $\beta(P) \not\in \{0,1,\infty\}$. We start with the case in which $P$ comes from an $\mathbb{F}_{q^2}$-rational place of $Z_3$.

\subsection{Computation of \texorpdfstring{$H(P)$}{H(P)} for the remaining rational places of \texorpdfstring{$Z_3$}{Z\_3}}

Assume that $P$ is a rational place of $Z_3$ such that $\beta(P) \not \in \{0,1,\infty\}$. From \cref{lem:beta_rational} this amounts to the case where $\beta(P)^{(q-1)/2}=-1$. Recall that if $P$ is $\mathbb{F}_{q^2}$-rational and has $\PP$-order $i$, then Remark \ref{rem:P-order_rational} implies that $i+1$ divides $q+1$. Hence either $i \in \{(q-1)/2,q\}$ or, since $q+1$ cannot be a multiple of three, $i < q/3-1=m-1$.
The following theorem uses the functions $f_j$ constructed in the previous section to complete the determination of $H(P)$ for $P$ rational in the aforementioned remaining cases.

\begin{theorem} \label{thm:rat}
Let $P$ be a rational place of $Z_3$ and assume that $\beta(P) \not \in \{0,1,\infty\}.$ Further, denote by $i$ the $\PP$-order of $P$. Then 
\[H(P)=\langle q,q+1,q-1+j(q-2),(i+1)(q-2) \, | \, 0 \le j < i\rangle, \text{ if $i < m-1$}\]
and
\[H(P)=\langle q,q+1,q-1+j(q-2) \, | \, 0 \le j < m\rangle, \text{ if $i \in \{(q-1)/2,q\}$}.\]
\end{theorem}

\begin{proof}
Since $P$ is a rational place and $Z_3$ is a maximal function field over $\mathbb{F}_{q^2}$, we know that $q,q+1 \in H(P)$. Now suppose first that $i \ge m-1$. By Remark \ref{rem:P-order_rational} and the discussion above, we see that this can only occur in case $i \in \{(q-1)/2,q\}$. Theorem \ref{thm:fi} implies the existence of functions $f_j$, for $0 \le j <m$, such that $f_j \in L((j+1)qP_\infty)$ and $v_P(f_j)=3j+2$, since $3j+2 \le 3(m-1)+2<q$. Then the functions $f_j/F_P^{j+1}$ have a pole only at $P$. Moreover $-v_P(f_j/F_P^{j+1})=(j+1)(q+1)-3j-2=q-1+j(q-2).$ This shows that the semigroup $H(P)$ contains the semigroup $H:=\langle q,q+1,q-1+j(q-2) \, | \, 0 \le j < m\rangle$. Note that this is precisely the same semigroup obtained in Theorem \ref{sembeta1}. Hence we have already shown there that $H$ has the correct number of gaps. 

Now suppose that $i<m-1$. Reasoning as above, we immediately conclude that for $0 \le j <i$, the integers $q-1+j(q-2)$ are in $H(P)$. Moreover, $v_P(f_i)=3i+3$, since $3i+3 < 3(m-1)+3=q$. Arguing similarly using the function $f_i/F_P^{i+1}$, we see that $(i+1)(q-2) \in H(P)$. Hence the semigroup $H_i:=\langle q,q+1,q-1+j(q-2),(i+1)(q-2) \, | \, 0 \le j < i\rangle$ is contained in $H(P)$. We are done if we can show that the semigroup $H_i$ has at most $g(Z_3)$ many gaps. At this point, also note that the semigroup from Theorem \ref{semp00} is exactly $H_1$. Therefore we will also show that $H_i$ has at most $g(Z_3)$ many gaps for the case $i=1$, thereby completing the proof of Theorem \ref{semp00} as well. 

We know $0 \in H$ and, using the given generators of $H$, one obtains that all integers in the set $\{j(q-2)+1,\dots,j(q+1)\}$ are in $H$ for any $j=1,\dots,i$. In addition, we have already proved that $(i+1)(q-2) \in H$, and adding $q-1$, $q$, and $q+1$ to the integers in $\{i(q-2)+1,\dots,i(q+1)\}$ gives that $\{(i+1)(q-2)+2,\dots,(i+1)(q+1)\} \subseteq H$.
Recall that $i+1$ divides $q+1$, since $P$ is a rational place of $Z_3$. We claim that for $k=0,\dots,\frac{q+1}{i+1}-1$ and all $j=1,\dots,i$ the sets $\{(k(i+1)+j)(q-2)+1,\dots,(k(i+1)+j)(q+1)\}$ are contained in $H$ as well as the integer $((k+1)(i+1))(q-2)$ and the set $\{(k+1)(i+1)(q-2)+2,\dots,(k+1)(i+1)(q+1)\}$. We have so far shown this for $k=0$. To complete the proof of the claim is then sufficient to notice that, if it is true for some $k-1 < \frac{q+1}{i+1}-1$, then adding $(i+1)(q-2)$ and the integers in $\{(i+1)(q-2)+2,\dots,(i+1)(q+1)\}$ shows that it is true for $k$ as well.

To count the number of gaps of $H$, write $\ell:=k(i+1)+j-1$. Note that, for any $k$, we find precisely $q - 2 - 3\ell$ gaps $\gamma$ such that $(k(i+1) + (j-1))(q+1) + 1 \leq \gamma \leq (k(i+1) + j)(q-2)$. Since $q - 2 - 3\ell$ is non-negative if and only if $\ell \leq \left\lfloor \frac{q-2}{3} \right\rfloor = m-1$, we obtain that the number of gaps of the semigroup $H_i$ is at most 
\begin{equation*}
    \sum_{\ell=0}^{m-1}(q-2-3\ell) = (q-2)m - \frac{3m(m-1)}{2} = \frac{q^2-q}{6},
\end{equation*}
which concludes the proof.
\end{proof}

Note that the last part of the proof of this theorem is similar to the proof of Theorem 5.5 of \cite{BMV2}. We now turn our attention to places of $\Zte$ that do not come from rational places of $Z_3$. 

\subsection{Weierstrass semigroups at the non-\texorpdfstring{$\F_{q^2}$}{F\_q\^2}-rational places}

In this section, we assume that $P$ is place of $\Zte$ that does not come from a rational place of $Z_3$. In order to compute the set of gaps $G(P)$ of $P$, the second family of special functions $g_\ell$ constructed in the previous subsection will be used. Hence the $\RR$-order of $P$ is going to play a role this time, and not only the $\PP$-order, which was sufficient in the case $P$ rational, in the previous subsection.

\begin{theorem}
     \label{thm:generic:semigroup}
Let $P$ be a place of $\Zte$ not coming from a rational place of $Z_3$. Suppose that the $\RR$-order of $\beta(P)$ is larger than or equal to $m-1$, then the set of gaps of $P$ is: 
\begin{equation*}
\label{eq:generic:gaps}
        G(P)=
        \{jq+k \mid j=0,\ldots,m-1,\ k=1,\ldots, q-2-3j\}.
\end{equation*}
\end{theorem}
\begin{proof}
We start by observing that the cardinality of the set $\{jq+k \mid j=0,\ldots,m-1,\ k=1,\ldots, q-2-3j\}$ is
$$\sum_{j=0}^{m-1}(q-2-3j)=m(q-2)-3\frac{(m-1)m}{2}=\frac{q^2-q}{6}=g(Z_3).$$
To show that each element $jq+k$ in this set is a gap at $P$, it is by \cref{lem:diffandgaps} enough to construct functions $h_{j,k}$ such that $v_P(h_{j,k})= jq + k - 1$ and $h_{j,k}\in L((m-1)(q+2)P_\infty)$. We consider the following cases.

\textbf{Case 1:} For $0\leq j\leq m-2$, $1\leq k\leq 3$, consider the functions
\begin{equation*}
h_{j,k} := 
\begin{cases} 
F_{P}^j \ &\mbox{if} \ k=1,\\[5pt]
F_{P}^j \cdot x_a \ &\mbox{if} \ k=2,\\[5pt]
F_{P}^j \cdot f_0 \ &\mbox{if} \ k=3.\\
\end{cases}
\end{equation*}

Note that if $k\equiv 0 \pmod{3}$, then $k\leq q-3-3j$, if $k\equiv 2 \pmod{3}$, then $k\leq q-4-3j$, and if $k\equiv 1 \pmod{3}$, then $k\leq q-2-3j$.

\textbf{Case 2:} For $0\leq j\leq m-2$, $4 \leq k\leq q-2-3j$, let $\ell:=\left\lfloor\frac{k}{3}\right\rfloor$ and consider the functions
\begin{equation*}
h_{j,k} := 
\begin{cases} 
F_{P}^j \cdot g_{\ell-2} \cdot f_0 \ &\mbox{if} \ k\equiv 0 \pmod{3},\\[5pt]
F_{P}^j \cdot g_{\ell-1} \ &\mbox{if} \ k\equiv 1 \pmod{3},\\[5pt]
F_{P}^j \cdot g_{\ell-1} \cdot x_a \ &\mbox{if} \ k\equiv 2 \pmod{3}.\\
\end{cases}
\end{equation*}

\textbf{Case 3:} For $j = m-1$, $k = q-2-3j = 1$, consider the function 
\begin{equation*}
    h_{j,k} := F_P^{m-1}.
\end{equation*}

We prove the theorem by checking separately each of the cases.

\textbf{Case 1:} Let $0\leq j\leq m-2$.
\begin{itemize}
    \item If $k=1$, we consider $h_{j,k} := F_{P}^j$ and we have $v_P(F_{P}^j) = jq$, and
        $$(F_{P}^j)_{\infty} = j(q+1)P_\infty \leq (m-2)(q+1)P_\infty \leq (m-1)(q+2)P_\infty.$$
    \item If $k=2$, we consider $h_{j,k} := F_{P}^j \cdot x_a$ and we have $v_P(F_{P}^j \cdot x_a) = jq + 1$, and
        $$(F_{P}^j \cdot x_a)_{\infty} = (j(q+1) + 2m) P_\infty \leq (m-1)(q+2)P_\infty$$
        since $2m \leq q$.
    \item If $k=3$, we consider $h_{j,k} := F_{P}^j \cdot f_0$ and we have $v_P(F_{P}^j \cdot f_0) = jq + 2$, and 
    $$(F_{P}^j \cdot f_0)_{\infty} = (j(q+1) + q) P_\infty \leq ((m-2) + (m-1)q)P_\infty \leq (m-1)(q+2)P_\infty.$$
\end{itemize}

\textbf{Case 2:} Let $0\leq j\leq m-2$, $4 \leq k\leq q-2-3j$, and $\ell:=\left\lfloor\frac{k}{3}\right\rfloor$. Note that $\ell \le \left\lfloor(q-2)/3\right\rfloor=m-1$ and hence that $v_P(g_{\ell-2})=3(\ell-2)+3$ and $v_P(g_{\ell-1})=3(\ell-1)+3$ by \cref{prop:functions:gell}, since we assumed that the $\RR$-order of $P$ is larger than or equal to $m-1$.
\begin{itemize}
    \item If $k \equiv 0 \pmod{3}$, $h_{j,k} := F_{P}^j \cdot g_{\ell-2} \cdot f_0$ and we have
    \begin{equation*}
        v_P(F_{P}^j \cdot g_{\ell-2} \cdot f_0) = jq + 3(\ell - 2) + 3 + 2 = jq + k - 1,
    \end{equation*}
    and, since $k\leq q - 2 - 3j$ and $k \equiv 0 \pmod{3}$ together imply that $k\leq q - 3 - 3j$, in this case 
    \begin{align*}
        (F_{P}^j \cdot g_{\ell-2} \cdot f_0)_\infty = (j(q+1) + (3\ell-2)m + q) P_\infty &= (j(q+1) + km - 2m + q) P_\infty \\
        &\leq (j + (q-3)m + m) P_\infty\\
        &\leq (m-2 + (q-2)m) P_\infty\\
        &= (m-1)(q+2)P_\infty.
    \end{align*}
    \item If $k \equiv 1 \pmod{3}$, $h_{j,k} := F_{P}^j \cdot g_{\ell-1}$ and we have
    \begin{equation*}
        v_P(F_{P}^j \cdot g_{\ell-1}) = jq + 3(\ell - 1) + 3 = jq + k - 1,
    \end{equation*}
    and, since $k\leq q - 2 - 3j$,
    \begin{align*}
        (F_{P}^j \cdot g_{\ell-1})_\infty = (j(q+1) + (3\ell + 1)m) P_\infty &= (j(q+1) + km) P_\infty \\
        &\leq (m-2 + (q-2)m) P_\infty\\
        &= (m-1)(q+2)P_\infty.
    \end{align*}
    \item If $k \equiv 2 \pmod{3}$, $h_{j,k} := F_{P}^j \cdot g_{\ell-1} \cdot x_a$ and we have
    \begin{equation*}
        v_P(F_{P}^j \cdot g_{\ell-1} \cdot x_a) = jq + 3(\ell - 1) + 3 + 1 = jq + k - 1,
    \end{equation*}
    and, since $k\leq q - 4 - 3j$,
    \begin{align*}
        (F_{P}^j \cdot g_{\ell-1} \cdot x_a)_\infty = (j(q+1) + (3\ell + 1)m + 2m) P_\infty &= (j(q+1) + (k+1)m) P_\infty \\
        &\leq (m-2 + (q-3)m) P_\infty\\
        &\leq (m-1)(q+2)P_\infty.
    \end{align*}
\end{itemize}

\textbf{Case 3:} If $j = m-1$, $k = q-2-3j = 1$, then $h_{j,k} := F_P^{m-1}$ and $v_P(F_P^{m-1}) = (m-1)q$, while 
\begin{equation*}
    (F_P^{m-1})_\infty = (m-1)(q+1)P_\infty \leq (m-1)(q+2)P_\infty.
\end{equation*}
Since all cases have been checked, the proof is hence concluded.
\end{proof}

We denote by $$G_{gen}:=\{jq+k \mid j=0,\ldots,m-1,\ k=1,\ldots, q-2-3j\},$$
the set of gaps determined in Theorem \ref{thm:generic:semigroup}. It is the generic set of gaps, that is to say, the set of gaps of all but finitely many places of $\Zte.$ Indeed, there are only finitely many places with $\RR$-order strictly less than $m-1.$

\begin{remark}
If the $\RR$-order $K$ of $P$ is strictly less than $m-1$, then $G(P)\neq G_{gen}$. To see this, it is sufficient to note that 
$$v_P(F_P^{m-K-2}g_K)+1=(m-K-2)q+3K+4 + 1 =(m-K-2)q+3K+5 \in G(P),$$
and this integer was not in $G_{gen}.$ Indeed, in $G_{gen}$, for $j=m-K-2$, the largest possible value of $k\equiv 2 \pmod{3}$ is $q-4-3(m-K-2)=3K + 2$.
\end{remark}

The following theorem deals with the non-generic non-rational places cases, that is, when the $\RR$-order of $P$ is at most $m-2$.
\begin{theorem}
    \label{thm:nonrational:weierstrass}
Let $P$ be a non-$\mathbb{F}_{q^2}$-rational place of $Z_3$. Let $i$ be the $\PP$-order and let $K$ be the $\RR$-order of $P$, and suppose that $K \le m-2$.
Then
\begin{equation*}
\label{eq:nonrational:weierstrass}
\begin{split}
G(P)=(G_{gen} \setminus \{(m-K-2-\ell (i+1))q+3K+4+ 3 \ell (i+1) \mid \ell=0,\ldots, \left\lfloor (m-K-2)/(i+1)\right\rfloor\} \\ \cup \{(m-K-2-\ell (i+1))q+3K+5 + 3\ell (i+1) \mid \ell=0,\ldots, \left\lfloor (m-K-2)/(i+1)\right\rfloor\}.
\end{split}
\end{equation*}
\end{theorem}

\begin{proof}

Let
\begin{equation*}
\begin{split}
\Gamma:=(G_{gen} \setminus \{(m-K-2-\ell (i+1))q+3K+4+ 3 \ell (i+1) \mid \ell=0,\ldots, \left\lfloor (m-K-2)/(i+1)\right\rfloor\} \\ \cup \{(m-K-2-\ell (i+1))q+3K+5 + 3\ell (i+1) \mid \ell=0,\ldots, \left\lfloor (m-K-2)/(i+1)\right\rfloor\}
\end{split}
\end{equation*}
be the putative set of gaps. It is clear that $|\Gamma| = |G_{gen}|$, as $\Gamma$ is obtained from $G_{gen}$ by replacing precisely $\left\lfloor (m-K-2)/(i+1)\right\rfloor + 1$ gaps $\gamma$ in $G_{gen}$ with $\gamma + 1$.  
Furthermore, for each $\gamma \in \Gamma$ such that $\gamma = jq + k$ with $0\leq j \leq m-1$ and $1\leq k\leq 3K+2$, we can use the same functions as for the generic case. Note that this covers, in particular, all cases for $m-K-1\leq j \leq m-1$.
Moreover, also the case $k=3K+3$ is already covered for all $0\leq j \leq m-2$, because in that case we only used the function $g_{K-1}$. Hence, we only need to treat the case $k\geq 3K+4$ and $0\leq j\leq m-K-2$. 

We define the following quantities: 
$$\cc:=\left\lfloor \frac{k-3K-4}{3(i+1)} \right\rfloor, \quad \dd:=\left\lfloor \frac{k-3K-4}{3} \right\rfloor, \quad s:=\dd-\cc(i+1) \quad \text{and} \quad r:=(k-3K-4)-3\dd.$$
From these definitions, we have in particular $0\leq s \leq i$ and $0\leq r\leq 2$, because they are the remainders of the integer division of $\dd$ by $i+1$ and of $k-3K-4$ by $3$, respectively. In fact, with these notations, we can write $k-3K-4=3(i+1)\cc+3s+r$.

To prove that each element $\gamma = jq + k$ of the putative set of gaps is an element of $G(P)$, we now construct functions 
\begin{equation*}
    \tilde{h}_{j,k}:=F_P^j \cdot g_K \cdot \hat{h}_{j,k}
\end{equation*}
such that $v_P(\hat{h}_{j,k})=k-3K-5$ and $\hat{h}_{j,k} \in L((m-2-K-j)qP_\infty)$. Indeed, in this way, the function $\tilde{h}_{j,k}$ is such that $v_P(\tilde{h}_{j,k}) = jq + k - 1$ and $\tilde{h}_{j,k} \in L((m-1)(q+2)P_\infty)$. This is the case because 
\begin{equation*}
    v_P(\tilde{h}_{j,k}) = jq + 3K + 4 + k - 3K - 5 = jq + k - 1,
\end{equation*}
and
\begin{equation*}
    -v_{P_\infty}(\tilde{h}_{j,k}) \leq j(q+1) + (3K+4)m + (m-2-K-j)q \leq (m-1)(q+2).
\end{equation*}
The last inequality holds since
\begin{align*}
    j(q+1) + (3K+4)m + (m-2-K-j)q &\leq (m-1)(q+2) \\
    \Leftrightarrow \ (m-2-K-j)q &\leq (m-1-j)(q+1) + (m-1) - 3Km - 4m\\
    \Leftrightarrow \ (m-2-K-j)q &\leq (m-1-j)q + (m-1-j) - Kq - q - 1\\ 
    \Leftrightarrow \ (m-1-j-K)q - q &\leq (m-1-j)q + (m-1-j) - Kq - q - 1\\
    \Leftrightarrow \ - Kq - q &\leq  (m-1-j) - Kq - q - 1\\
    \Leftrightarrow \ j &\leq  m - 2,
\end{align*}
and $j\leq m-2$ is satisfied for all the values of $j$ that we are considering, since we are only considering values such that $0\leq j \leq m - K - 2$, by the discussion above.

We are hence left with constructing functions $\hat{h}_{j,k}$ that satisfy the discussed properties. We do this distinguishing several cases. Note that, mutatis mutandis, this construction coincides with the one in the proof of \cite[Theorem 6.12]{BMV2}.
\begin{enumerate}
    \item If $s>0$, we define:
\begin{equation*}
        \hat{h}_{j,k}:=\begin{cases}
            f_i^\cc \cdot f_{s-1} \quad &\mbox{if} \ r=0,\\[5pt]
            f_i^\cc \cdot f_{s-1} \cdot x_a \quad &\mbox{if} \ r=1,\\[5pt]
            f_i^\cc \cdot f_{s-1} \cdot f_0 \quad &\mbox{if} \ r=2.\\
        \end{cases}
    \end{equation*}
    \item If $s=0$, we define instead:
\begin{equation*}
        \hat{h}_{j,k}:=\begin{cases}
            f_i^{\cc-1}\cdot f_{i-1} \cdot f_0 \cdot x_a \quad &\mbox{if} \ r=0,\\[5pt]
            f_i^\cc \quad &\mbox{if} \ r=1,\\[5pt]
            f_i^\cc \cdot x_a \quad &\mbox{if} \ r=2.
        \end{cases}
    \end{equation*}
\end{enumerate}

We start by observing that, from these definitions, it is clear that the functions $\hat{h}_{j,k}$ only have poles at $P_\infty$.
We now show that they also satisfy the desired properties that $v_P(\hat{h}_{j,k})=k-3K-5$ and $\hat{h}_{j,k} \in L((m-2-K-j)qP_\infty)$.

\begin{enumerate}
    \item Case $s>0$.
    \begin{enumerate}
        \item If $r=0$, then $k = 3\dd + 3K + 4$. In particular $k \equiv 1 \pmod{3}$ and hence $k\leq q - 2 - 3j$, by definition of the putative set of gaps. We have 
        \begin{equation*}
            v_P(f_i^\cc \cdot f_{s-1}) = 3\cc(i+1) + 3(s-1) + 2 = 3\dd - 1 = k - 3K - 5,
        \end{equation*}
        and
        \begin{align*}
            -v_{P_\infty}(f_i^\cc \cdot f_{s-1}) &= \cc(i+1)q + sq = \dd q = \frac{k - 3K - 4}{3}\cdot q \\
            &\leq \frac{q - 2 - 3j - 3K - 4}{3}\cdot q = (m - 2 - K - j)q.
        \end{align*}
        \item If $r=1$, then $k = 3\dd + 3K + 5$. In particular $k \equiv 2 \pmod{3}$ and hence $k\leq q - 4 - 3j$, by definition of the putative set of gaps. We have
        \begin{equation*}
            v_P(f_i^\cc \cdot f_{s-1} \cdot x_a) = 3\cc(i+1) + 3(s-1) + 2 + 1 = 3\dd = k - 3K - 5,
        \end{equation*}
        and
        \begin{align*}
            -v_{P_\infty}(f_i^\cc \cdot f_{s-1} \cdot x_a) &= \cc(i+1)q + sq + 2m = \dd q + 2m = \frac{k - 3K - 5}{3}\cdot q + 2m \\
            &\leq \frac{q - 4 - 3j - 3K - 5}{3}\cdot q + 2m = (m - 2 - K - j)q - m.
        \end{align*}
        \item If $r=2$, then $k = 3\dd + 3K + 6$. In particular $k \equiv 0 \pmod{3}$ and hence $k\leq q - 3 - 3j$, by definition of the putative set of gaps. We have
        \begin{equation*}
            v_P(f_i^\cc \cdot f_{s-1} \cdot f_0) = 3\cc(i+1) + 3(s-1) + 2 + 2 = 3\dd + 1 = k - 3K - 5,
        \end{equation*}
        and
        \begin{align*}
            -v_{P_\infty}(f_i^\cc \cdot f_{s-1} \cdot f_0) &= \cc(i+1)q + sq + q = \dd q + 3m = \frac{k - 3K - 6}{3}\cdot q + 3m \\
            &\leq \frac{q - 3 - 3j - 3K - 6}{3}\cdot q + 3m = (m - 2 - K - j)q.
        \end{align*}
    \end{enumerate}
    \item Case $s=0$.
    \begin{enumerate}
        \item If $r=0$, then $k = 3\dd + 3K + 4 = 3\cc (i+1) + 3K + 4$. In particular $k \equiv 1 \pmod{3}$ and, in this case, $k\leq q - 5 - 3j$, by definition of the putative set of gaps. Indeed, the only values of $k = 3\cc (i+1) + 3K + 4$ for which $jq + k\in \Gamma$ are those such that  $3\cc (i+1) \geq m - K - 1$, i.e., $k \geq m + 2K + 3$. However, since $q - 2 - 3j \leq 3K + 4$ and $m + 2K + 3 \leq 3K + 4$ if and only if $K\geq m-1$, this means that there are no values of $k$ in the putative set of gaps that are equal to $q-2-3j$, since we are assuming $K\leq m-2$. Therefore, we have that in this case $k\leq q - 5 - 3j$. Hence, we have
        \begin{equation*}
            v_P(f_i^{\cc-1}\cdot f_{i-1} \cdot f_0 \cdot x_a) = 3(i+1)(\cc - 1) + 3(i-1) + 2 + 2 + 1 = 3\dd - 1 = k - 3K - 5,
        \end{equation*}
        and
        \begin{align*}
            -v_{P_\infty}(f_i^{\cc-1}\cdot f_{i-1} \cdot f_0 \cdot x_a) &= (\cc-1)(i+1)q + iq + q + 2m = \dd q + 2m = \frac{k - 3K - 4}{3}\cdot q + 2m \\
            &\leq \frac{q - 5 - 3j - 3K - 4}{3}\cdot q + 2m = (m - 2 - K - j)q - m.
        \end{align*}
        \item If $r=1$, then $k = 3\dd + 3K + 5 = 3\cc (i+1) + 3K + 5$. In particular $k \equiv 2 \pmod{3}$ and hence $k\leq q - 4 - 3j$, by definition of the putative set of gaps. We have
        \begin{equation*}
            v_P(f_i^\cc) = 3\cc(i+1) = 3\dd = k - 3K - 5,
        \end{equation*}
        and
        \begin{align*}
            -v_{P_\infty}(f_i^\cc) &= \cc(i+1)q = \dd q = \frac{k - 3K - 5}{3}\cdot q \\
            &\leq \frac{q - 4 - 3j - 3K - 5}{3}\cdot q + 2m = (m - 2 - K - j)q - q.
        \end{align*}
        \item If $r=2$, then $k = 3\dd + 3K + 6 = 3\cc (i+1) + 3K + 6$. In particular $k \equiv 0 \pmod{3}$ and hence $k\leq q - 3 - 3j$, by definition of the putative set of gaps. We have
        \begin{equation*}
            v_P(f_i^\cc \cdot x_a) = 3\cc(i+1) + 1 = 3\dd + 1 = k - 3K - 5,
        \end{equation*}
        and
        \begin{align*}
            -v_{P_\infty}(f_i^\cc \cdot x_a) &= \cc(i+1)q + 2m = \dd q + 2m = \frac{k - 3K - 6}{3}\cdot q + 2m \\
            &\leq \frac{q - 3 - 3j - 3K - 6}{3}\cdot q + 2m = (m - 2 - K - j)q - m.
        \end{align*}
    \end{enumerate}
\end{enumerate}
Since we have considered all possible cases, we have constructed all the gaps in the putative set $\Gamma$, thereby concluding the proof.
\end{proof}

\section{The automorphism group of \texorpdfstring{$Z_3$}{Z\_3}}

In this section we aim to apply the obtained results on Weierstrass semigroups to compute the full automorphism group of $Z_3$. As used previously, we know that $\aut(Z_3)$ contains the subgroup
$$G=\{(x,y)\mapsto (x+a,\pm y+b) \mid a^q+a=0, \quad p(b)=0\}$$
of order $2q^3/3$.
The structure of this group can be described in more detail by first identifying the subgroup
$$S_3:=\{(x,y)\mapsto (x+a,y+b) \mid a^q+a=0, \quad p(b)=0\},$$
which is an elementary abelian $3$-group of order $q^2/3$.
An element of order two in $\aut(Z_3)$ is given by
$$\sigma_2: (x,y) \mapsto (x,-y).$$
Then the group $G$ is generated by $S_3$ and $\sigma_2$. More precisely, it is the semidirect product $G=S_3 \rtimes C_2$, where $C_2=\langle \sigma_2 \rangle$ is cyclic of order two. The group $G$ clearly fixes $P_\infty$. 

We want to show that $\aut(Z_3)$ is simply $G$, unless $t=1$. The case $t=1$ will not be considered in the following, the reason being that in this case $Z_3$ is elliptic, and hence $\aut(Z_3)$ is infinite, see \cite[Theorem 11.94 (i)]{HKT}.

\begin{lemma} \label{fixpi}
Let $q=3^t$ with $t \geq 2$. Then $\aut(Z_3)$ fixes $P_\infty$. In particular $\aut(Z_3)=Q \rtimes C$, where $Q$ is a $3$-group of order divisible by $q^2/3$ and $C$ is a cyclic group of order divisible by $2$, but not by $3$.     
\end{lemma}

\begin{proof}
Let $O_\infty$ be the $\aut(Z_3)$-orbit containing $P_\infty$. Clearly $O_\infty$ contains only $\mathbb{F}_{q^2}$-rational places as $\aut(Z_3)$ is defined over $\mathbb{F}_{q^2}$, see e.g. \cite[Lemma 2.4]{BMT}. To show that $O_\infty=\{P_\infty\}$ it is hence enough to show that $H(P) \ne H(P_\infty)$ for any rational place $P$ of $Z_3$ distinct from $P_\infty$. 
From Theorem \ref{sempinf} we see that $2q/3 \in H(P_\infty)$. On the other hand, by Theorems \ref{semp00}, \ref{sembeta1} and \ref{thm:rat}, we see that $2q/3$ is a gap for any other rational place $P$ of $Z_3$. Hence $H(P) \ne H(P_\infty)$ if $P \neq P_\infty$.

The rest of the statement follows by \cite[Theorem 11.49]{HKT} and the fact that $\aut(Z_3)$ contains $G$.
\end{proof}

To prove that $\aut(Z_3)=G$ we are left with showing that $Q=S_3$ and $C=C_2$. The next two lemmas deal with proving that $Q=S_3$.

\begin{lemma} \label{orbp00}
Let $q=3^t$ with $t \geq 2$. Then the $\aut(Z_3)$-orbit containing $P_{(0,0)}$ coincides with the set $\Omega:=\{P=P_{(a,b)} \mid a,b \in \mathbb{F}_{q^2}, \beta(P)=0\}$.  
\end{lemma}

\begin{proof}
From the previous lemma we know that $\aut(Z_3)=\aut(Z_3)_{P_\infty}$, while $G \subseteq \aut(Z_3)$ acts transitively on the set $\Omega$ of $\mathbb{F}_{q^2}$-rational places $P$ satisfying $\beta(P)=0$, which includes the place $P_{(0,0)}$. In fact, directly from the definition of the subgroup $S_3$ of $G$ of order $q^2/3$, one can derive that it acts sharply transitively on $\Omega$. This implies that the $\aut(Z_3)$-orbit containing $P_{(0,0)}$ contains $\Omega$. To complete the proof we need to prove that if $P$ is $\mathbb{F}_{q^2}$-rational and $\beta(P) \ne 0$ then $P$ cannot be in the same $\aut(Z_3)$-orbit as $P_{(0,0)}$. To do so, it suffices to show that for any such $P$, $H(P) \ne H(P_{(0,0)})$. We do not need to consider the case $P=P_\infty$, as we already know from the previous lemma that $\aut(Z_3)=\aut(Z_3)_{P_\infty}$.
From \cref{semp00}, we see that $2q-3 \in G(P_{(0,0)})$. However, Theorems \ref{thm:rat} and Theorem \ref{sembeta1} imply that $2q-3 \in H(P)$, and hence $H(P) \ne H(P_{(0,0)})$. This completes the proof.
\end{proof}

\begin{lemma}
Let $q=3^t$ with $t \geq 2$. Then $\aut(Z_3)=S_3 \rtimes C$, where $S_3$ is the elementary abelian $3$-group of order $q^2/3$ given by
$$S_3=\{(x,y) \mapsto (x+a,y+b) \mid a^q+a=p(b)=0\},$$
and $C$ is a cyclic group of order divisible by $2$, but not by $3$. 
\end{lemma}

\begin{proof}
From Lemma \ref{fixpi}, we know that $\aut(Z_3)$ fixes $P_\infty$, $\aut(Z_3)=Q \rtimes C$, where $Q$ is a $3$-group of order divisible by $q^2/3$ (containing $S_3$) and $C$ is a cyclic group of order divisible by $2$, but not by $3$.   
Furthermore, from Lemma \ref{orbp00} we know that the $\aut(Z_3)$-orbit containing $P_{(0,0)}$ is $\Omega=\{P=P_{(a,b)} \mid a,b \in \mathbb{F}_{q^2}, \beta(P)=0\}$. This means that this orbit has cardinality $q^2/3$.

Since $Z_3$ is a maximal function field, it has $3$-rank zero. Hence from \cite[Lemma 11.129]{HKT} and the fact that $Q$ fixes $P_\infty$, we may conclude that $Q$ acts with orbits of length $|Q|$ on $\Omega$. This implies that $|Q|$ needs to divide $|\Omega|=q^2/3$. On the other hand we already observed that $|Q|$ is divisible by $q^2/3$. Hence $|Q|=q^2/3$, so that $Q=S_3$ as claimed.
\end{proof}

Finally the following proposition shows that if $t \geq 2$ then not only $Q=S_3$, but also $C=C_2$, implying that $\aut(Z_3)=G$.

\begin{proposition}
Let $q=3^t$ with $t \geq 2$. Then $\aut(Z_3)=G=S_3 \rtimes C_2$. 
\end{proposition}

\begin{proof}
According to the previous lemmas, it is enough to prove that $C=C_2$. Since $C$ is cyclic, $C_2$ is trivially a normal subgroup of $C$. This implies not only that $C$ fixes $P_\infty$ (which we knew already) but also that $C$ acts on the set of remaining fixed places of $C_2$. Apart from $P_\infty$, the fixed places of $C_2$ are the places $P_{(a,0)}$ where $a^q+a=0$. Hence if $C=\langle \gamma \rangle$ then from Equation \eqref{divy1}, we see that $\gamma(y)=\lambda y$ for some constant $\lambda \in \mathbb{F}_{q^2}^*$. Since by \cref{orbp00} $\aut(Z_3)$ acts on the set of places with $\beta(P)=0$, we also know that $\gamma(p(y))=\mu p(y)$ for some nonzero constant $\mu$. Putting everything together we obtain
$$\mu \sum_{i=0}^{t-1}y^{3^i}=\mu p(y)=\gamma(p(y))=p(\lambda y)=\sum_{i=0}^{t-1} (\lambda y)^{3^i}.$$
Using $t \ge 2$, we deduce from this that $\lambda=\mu \in \mathbb{F}_3^*$ and hence $\gamma(y)=-y$ (the case $\gamma(y)=y$ cannot happen, as $C$ contains $C_2$).

To complete the proof we wish to show that $\gamma(x)=x$, so that $C=\langle \gamma \rangle =C_2$. To do so, note that since $x \in L(2mP_\infty)$ and $\gamma$ fixes $P_\infty$, we have $\gamma(x) \in L(2mP_\infty)$. From the proof of Theorem \ref{sempinf}, $L(2mP_\infty)=\langle 1,x \rangle$. 
This shows that $\gamma(x)=ax+b$ for some $a,b \in \mathbb{F}_{q^2}$. Hence
$$0=\gamma(x^q+x)+\gamma(p(y)^2)=\gamma(x)^q+\gamma(x)+p(y)^2=(ax+b)^q+(ax+b)-(x^q+x).$$
The equation above implies that $(a-1)^qx^q+(a-1)+(b^q+b)=0$. Hence $a=1$ and so $\gamma(x)=x+b$ for some $b$ satisfying $b^q+b=0$. We claim that $b=0$. Assume by contradiction $b \ne 0$ then $\gamma^3(x)=x$ (note that $\gamma(x) \ne x$) and $\gamma^2(y)=y$. This shows that the order of $\gamma$ is $6$, which is not possible as $C=\langle \gamma \rangle$ has order not divisible by $3$. This shows that $b=0$ and hence $\gamma(x)=x$ as claimed. 
\end{proof}

\section*{Acknowledgments}
This work was supported by a research grant (VIL”52303”) from Villum Fonden. The third author was partially supported by an NWO Open Competition ENW – XL grant (project "Rational points: new dimensions").

\end{document}